\documentclass[a4paper]{amsart}

\usepackage[foot]{amsaddr}

\usepackage{xcolor}

\usepackage{amsfonts}
\usepackage{amssymb}
\usepackage{amsmath}
\usepackage{hyperref}
\usepackage{mathrsfs}
\usepackage{centernot}
\usepackage{mathdots}
\usepackage{stmaryrd}
\usepackage[all]{xy}
\usepackage{tikz}

\newtheorem{thm}{Theorem}[section]
\newtheorem{lem}[thm]{Lemma}
\newtheorem{prp}[thm]{Proposition}
\newtheorem{cor}[thm]{Corollary}

\newtheorem{que}[thm]{Question}

\newtheoremstyle{roman} 
    {8.0pt plus 2.0pt minus 4.0pt}                    
    {8.0pt plus 2.0pt minus 4.0pt}                    
    {\normalfont}                
    {}                           
    {\bfseries}                  
    {.}                          
    {5pt plus 1pt minus 1pt}     
    {}  

\theoremstyle{roman}

\newtheorem{example}[thm]{Example}
\newtheorem{remark}[thm]{Remark}

\theoremstyle{plain}

\newcommand{\rem}[1]{}

\newcommand{\C}{\mathbb{C}}

\newcommand{\N}{\mathbb{N}}
\newcommand{\Q}{\mathbb{Q}}
\newcommand{\R}{\mathbb{R}}
\newcommand{\Z}{\mathbb{Z}}

\newcommand{\HH}{{\mathrm{H}}}

\newcommand{\fraka}{{\mathfrak{a}}}

\newcommand{\frakm}{{\mathfrak{m}}}

\newcommand{\frakp}{{\mathfrak{p}}}

\newcommand{\calB}{{\mathcal{B}}}

\newcommand{\calG}{{\mathcal{G}}}

\newcommand{\calO}{{\mathcal{O}}}

\newcommand{\bbH}{{\mathbb{H}}}

\newcommand{\bfG}{{\mathbf{G}}}

\newcommand{\vphi}{\varphi}

\newcommand{\derives}{\Longrightarrow}

\newcommand{\suchthat}{\,:\,}
\newcommand{\where}{\,|\,}

\newcommand{\quo}[1]{\overline{#1}}

\newcommand{\SMatII}[4]{\left[\begin{array}{cc} {#1} & {#2} \\ {#3} &
{#4} \end{array}\right]}
\newcommand{\smallSMatII}[4]{\left[\begin{smallmatrix} {#1} & {#2} \\ {#3} &
{#4} \end{smallmatrix}\right]}

 \DeclareMathOperator{\Aut}{Aut}
\DeclareMathOperator{\Cent}{Cent} %
\DeclareMathOperator{\End}{End} %
\DeclareMathOperator{\Gal}{Gal} %
\DeclareMathOperator{\Hom}{Hom} %
\DeclareMathOperator{\id}{id} %
\DeclareMathOperator{\Int}{Int} %
\DeclareMathOperator{\Jac}{Jac} %
\DeclareMathOperator{\Max}{Max} %
\DeclareMathOperator{\Spec}{Spec} %
\newcommand{\nMat}[2]{\mathrm{M}_{#2}(#1)}

\newcommand{\uGL}{{\mathbf{GL}}}
\newcommand{\uPGL}{{\mathbf{PGL}}}

\newcommand{\uU}{{\mathbf{U}}}

\newcommand{\uGm}{{\mathbf{G}}_{\mathbf{m}}}

\newcommand{\nGm}[1]{{\mathbf{G}}_{\mathbf{m},{#1}}}

\newcommand{\uAut}{{\mathbf{Aut}}}

\newcommand{\et}{\mathrm{\acute{e}t}}

\newcommand{\fpqc}{\mathrm{fpqc}}

\newcommand{\units}[1]{{#1^\times}}

\newcommand{\hen}[1]{#1^{\mathrm{h}}} 
\newcommand{\sh}[1]{#1^{\mathrm{sh}}} 

\newcommand{\inv}{\mathrm{inv}}

\numberwithin{equation}{section}

\title{Orders that are \'Etale-Locally Isomorphic}

\date{\today}

\author{Eva Bayer-Fluckiger$^\dagger$}
\email{eva.bayer@epfl.ch}
\author{Uriya A.\ First$^*$}
\email{uriya.first@gmail.com}
\author{Mathieu Huruguen$^\dagger$}
\email{mathieu.huruguen@epfl.ch}

\address{$^\dagger$Department of Mathematics, \'Ecole Polytechnique F\'ed\'erale de Lausanne}
\address{$^*$Department of Mathematics, University of Haifa}

\thanks{This research was supported by a Swiss National Science Foundation grant \#200021\_163188.}


\begin{document}

\maketitle

%

\section{Introduction}

Let $R$ be a regular local ring with fraction field $F$
and $\calG$ a  group $R$-scheme. We consider the restriction map
$$\HH^1_\et(R,\calG)\to \HH^1_\et(F,\calG)$$
and raise the {\em injectivity question } for the $R$-group scheme $\calG$: is the restriction map injective? That is, if two $\calG$-torsors defined over $R$ become isomorphic over $F$, are they isomorphic over $R$?

\medskip

We will prove that, if $R$ is a semilocal Dedekind domain and $A$ is a hereditary $R$-order in a central simple $F$-algebra, then the restriction map is injective when $\calG$ is   the $R$-group scheme 
$\uAut_R(A)$ of automorphism  of $A$, see Theorem \ref{TH:Aut-A-works}, 
and has trivial kernel when     
$\calG=\uPGL_1(A):=\uGL_1(A)/\nGm{R}$, see Theorem \ref{TH:PGL-A-works}. 
The former case  is equivalent to saying that two hereditary orders
in a central simple $F$-algebras that become isomorphic after base change to $F$
and a faithfully
flat \'etale $R$-algebra are already isomorphic. However, this fails for
hereditary orders in simple non-central algebras, see Example~\ref{EX:order-I}.

In contrast, we will show in Section~\ref{sec:counterexamples} that the injectivity question has a negative answer in general for the $R$-group schemes $\uAut_R(A,\sigma)$, where $A$ and $R$ are as above and $\sigma:A\to A$ is an involution fixing $R$.
Equivalently, it can happen that two  hereditary orders with involution which
become isomorphic after base change to $F$ and  a faithfully
flat \'etale $R$-algebra are   non-isomorphic. This cannot happen if $(A,\sigma)$ is Azumaya over $R$, though;
see \cite{Panin_05_purity_for_mult}.

\medskip

To put these results in perspective, recall the famous conjecture of Grothendieck and Serre \cite[Remarque $3$ pp. 26-27]{Groth_58_torsion_homology}, \cite[Remarque $1.11.a$]{Groth_68_Brauer_II}, 
\cite[Remarque p. 31]{Serre_58_fibered_spaces}, which stipulates that the injectivity question has a positive answer if $R$ is a regular local ring and $\calG$ is a reductive group $R$-scheme. This conjecture is still open in full  generality, but many cases have been settled. 
For example, see Nisnevich \cite{Nisne_84_principle_bundles} when $R$ is a discrete valuation ring, Colliot-Th\'el\`ene and Sansuc \cite{Colliot_87_flasque_tori_app} when $G$ is a torus over $R$, and Fedorov and Panin \cite{Fedor_15_Groth_Serre_conj}, 
\cite{Panin_17_GS_conj_III}, when the ring $R$ contains a field $k$.

In \cite{Bayer_17_rational_iso_herm}, the injectivity question 
is answered on the positive for group $R$-schemes of the form
$\uU(A,\sigma)$ with $A$ a hereditary $R$-order.
This led
the first two   authors   to formulate an extension of the Grothendieck-Serre conjecture, see \cite[Question~6.4]{Bayer_17_rational_iso_herm}. They
ask, in the case where $R$ is a semilocal Dedekind domain, whether the injectivity question has a positive answer for a certain family of $R$-group schemes larger than the family of reductive groups, which, 
loosely speaking,   arise from   Bruhat--Tits theory.
It seems very likely that the group schemes considered above fall into the family considered
in \cite[\S6]{Bayer_17_rational_iso_herm}, but more work is required to verify this.
Provided this holds,
the positive results of the present paper answer \cite[Question 6.4]{Bayer_17_rational_iso_herm} on the positive
for certain group schemes, and on the negative for others.
We elaborate  about this 
in Section~\ref{sec:remarks}.

\section{Hereditary Orders}

\label{sec:orders}

	Throughout this section,	
	$R$ denotes a Dedekind domain with fraction field
    $F$.
    For  $ \frakp\in \Max R$, let $R_\frakp$ denote the localization
    of $R$ at $\frakp$, and let $\hen{R_\frakp}$ and $\sh{R_\frakp}$ denote the henselization
    and strict henselization of $R_\frakp$, respectively. The corresponding  fraction fields
    are denoted   $\hen{F_\frakp}$ and $\sh{F_\frakp}$.  
    Unadorned tensors are always assumed to be over $R$.  
    
    Recall that an $R$-order in an $F$-algebra 
    $E$ is an $R$-subalgebra $A$ which contains an $F$-basis of
    $E$ and is finitely generated as an $R$-module.
    Equivalently, an $R$-algebra $A$ is an $R$-order in some $F$-algebra 
    $E$ (necessarily isomorphic to $A\otimes F$) if
    and only if $A$ is   $R$-torsion-free and finitely generated as an $R$-module.
    
    As usual, a ring $A$ is called \emph{hereditary}
    if its one-sided ideals are projective $A$-modules.
    Notable examples of hereditary rings include maximal $R$-orders in central simple
    $F$-algebras. Thus, every central simple $F$-algebra contains a hereditary
    order. See \cite[Chapter~9]{Reiner_03_maximal_orders} for an extensive discussion.

\medskip

	We shall prove:
	
	\begin{thm}\label{TH:Aut-A-works}
		Suppose $R$ is a semilocal Dedekind domain,
		let $A$ be a hereditary $R$-order in a central simple
		$F$-algebra and let $A'$ be any $R$-order.
		If $A$ and $A'$ become isomorphic after tensoring with  $F$ and  some faithfully
		flat \'etale $R$-algebra,
		then $A\cong A'$ as $R$-algebras.
	\end{thm}
	
	As a consequence, we get:
	
	\begin{cor}\label{CR:Aut-A-works}
		Suppose $R$ is a semilocal Dedekind domain and
		let $A$ be a hereditary $R$-order in a central simple
		$F$-algebra.
		Then the restriction map
		\[
		\HH^1_\et(R,\uAut_R(A))\to \HH^1_\et(K,\uAut_R(A))
		\]
		is injective. Here, $\uAut_R(A)$ denotes the group $R$-scheme
		whose $S$-points are given
		by $\uAut_R(A)(S)=\Aut_S(A\otimes  S)$.
	\end{cor}	
	
	\begin{proof}
		The cohomology set $\HH^1_{\et}(R,\uAut_R(A))$ classifies isomorphism classes
		of $R$-orders which become isomorphic to $A$ after tensoring
		with a faithfully flat \'etale $R$-algebra.
		By Theorem~\ref{TH:Aut-A-works}, it is enough to show
		that any such $R$-order $A'$ is also hereditary.
		
		Let $S$ be a faithfully
		flat \'etale $R$-algebra such that $A\otimes S\cong A'\otimes S$
		as $S$-algebras.
		Since $S$ is faithfully flat and \'etale over $R$,
		for all $\frakp\in\Max R$,
		there exists an $R$-algebra homomorphism 
		$S\to \sh{R_\frakp}$.
		Thus, $A\otimes \sh{R_\frakp}\cong A'\otimes \sh{R_\frakp}$
		for all $\frakp\in\Max R$.
		We shall   see below, in Corollary~\ref{CR:hereditary-descent},
		that this implies that $A'$ is hereditary.
	\end{proof}

	The proof of Theorem~\ref{TH:Aut-A-works} will be done
	in two steps. First, we will prove the theorem when $R$ is a henselian discrete valuation ring
	(abbrev.: DVR). This
	step will rely heavily on the structure theory of hereditary orders.
	Then, the general case will be deduced by means of patching.
	
	We will also show that the theorem fails for 
	hereditary $R$-orders in the larger class of finite-dimensional (not necessarily central) simple $F$-algebras.

\subsection{Preliminary Results}

	\begin{lem}\label{LM:hereditary-descent}
    	Assume that $R$ is a DVR, and let $R'$ be a DVR which is also a
    	faithfully flat $R$-algebra. Denote the maximal ideals of $R$, $R'$ by $\frakm$, $\frakm'$
    	respectively, and suppose that $k':=R'/\frakm'$ is a separable algebraic field extension of $k:=R/\frakm$.
    	Let $A$ be an $R$-order. Then
    	\begin{enumerate}
    		\item[(i)] $\Jac(A)\otimes R'=\Jac(A\otimes R')$ and
    		\item[(ii)] $A$ is hereditary if and only if $A\otimes R'$ is hereditary.
    	\end{enumerate}
    \end{lem}

    \begin{proof}
    	Write $A'=A\otimes R'$, $J=\Jac(A)$, $J'=\Jac(A')$ and view $A$ as a subring of $A'$.
    	Since $R'$ is a flat $R$-module, the map $J\otimes_RR'\to A'$ is injective, hence we may identify
    	$J\otimes_RR'$ with $JR'$.
    	
    	(i)  We need to show that $J'=JR'$.
    	By  \cite[Theorem~6.15]{Reiner_03_maximal_orders}, there is $n\in\N$ such that $J^n\subseteq \frakm A\subseteq J$
    	and $J'^n\subseteq \frakm'A'\subseteq J'$.
    	In particular, we may view $A/J$ as a $k$-algebra,
    	and therefore,
    	$A'/JR'\cong (A/J)\otimes_RR'\cong (A/J)\otimes_k k'$.
    	Since $A/J$ a semisimple  finite-dimensional  $k$-algebra and $k'$ is separable over $k$,
    	the ring $(A/J)\otimes_k k'$ is semisimple, and therefore $J'\subseteq JR'$.
    	On the other hand, $(JR')^n=J^nR'\subseteq\frakm AR'=\frakm A'\subseteq \frakm'A'\subseteq J'$, hence
    	$JR'\subseteq J'$, because $J'$ is semiprime. We conclude that $J'=JR'$.

    	(ii) By \cite[p.~5]{Auslan_60_maximal_orders}, it is enough to prove that $J:=\Jac(A)$ is projective as a right $A$-module
    	if and only if  $J':=\Jac(A')$ is projective as a right $A'$-module.
    	The direction ($\derives$) follows from (i) and the other direction
    	follows from \cite[Proposition~4.80(2)]{Lam_99_modules_and_rings} (the proof in \cite{Lam_99_modules_and_rings} is given for $R$-modules but extends
        verbatim to $A$-modules once
        noting that $\Hom_{A\otimes S}(M\otimes S,N\otimes S)\cong \Hom_A(M,N)\otimes S$ whenever $M$ is a finitely presented
        $A$-module and $S$ is a flat $R$-algebra; see \cite[Theorem~2.38]{Reiner_03_maximal_orders} for a proof of the latter).
    \end{proof}

    \begin{cor}\label{CR:hereditary-descent}
    	Let $A$ be an $R$-order. Then $A$ is hereditary if and only if $A\otimes  \hen{R_\frakp}$ is hereditary
    	for all $ \frakp\in\Max R$ if and only if $A \otimes  \sh{R_\frakp}$ is hereditary
    	for all $  \frakp\in\Max R$.
    \end{cor}

    \begin{proof}
    	By \cite[p.~8]{Auslan_60_maximal_orders},
    	$A$ is hereditary if and only if $A\otimes  R_\frakp$ is hereditary for all $ \frakp\in\Max R$.
    	Now use Lemma~\ref{LM:hereditary-descent}(ii).
    \end{proof}

\subsection{The Henselian Case}
\label{subsec:henselian-case}

    In this subsection, we assume   that
    $R$ is a henselian DVR with maximal ideal $\frakm$ and residue field $k$.
    We shall deduce Theorem~\ref{TH:Aut-A-works} in this special case as a consequence of the structure
    theory of hereditary $R$-orders, which we now recall.

    \begin{thm}[{\cite[\S12]{Reiner_03_maximal_orders}}]
    \label{TH:structure-hered-div}
    Let $D$ be a finite dimensional division $F$-algebra. Then the additive valuation $\nu_F$  of $F$
    extends uniquely to a discrete additive valuation $\nu_D$ on $D$. Furthermore,
    \[
    \calO_D:=\{x\in D \suchthat \nu_D(x)\geq 0\}
    \]
    is the unique maximal $R$-order in $D$.
    \end{thm}

    The unique maximal right (and left) ideal of $\calO_D$ is denoted $\frakm_D$
    and we write $k_D=\calO_D/\frakm_D$. The ring $k_D$ is a finite dimensional  
    division $k$-algebra, the center of which may be strictly larger than $k$.


\medskip

    Given a ring $A$ and ideals $(\fraka_{ij})_{i,j}$, we let
    \[
    \left[
    \begin{array}{ccc}
    (\fraka_{11}) & \dots & (\fraka_{1r}) \\
    \vdots & & \vdots \\
    (\fraka_{r1}) & \dots & (\fraka_{rr})
    \end{array}
    \right]^{(n_1,\dots,n_r)}
    \]
    denote the set of block matrices $(X_{ij})_{ i,j\in\{1,\dots, r\}}$
    for which $X_{ij}$ is an $n_i\times n_j$ matrix with entries in $\fraka_{ij}$.
    If $D$ is a  finite dimensional division $F$-algebra and $n_1,\dots,n_r$ are natural numbers, we write
    \begin{align*}
    \calO_D^{(n_1,\dots,n_r)}&=
    \left[
    \begin{array}{cccc}
    	(\calO_D) & (\frakm_D) & \dots & (\frakm_D) \\
    	\vdots & \ddots & \ddots & \vdots \\
    	\vdots & & \ddots & (\frakm_D) \\
    	(\calO_D) &\dots  & \dots & (\calO_D)
    \end{array}
    \right]^{(n_1,\dots,n_r)}\ ,
    \\
    J_D^{(n_1,\dots,n_r)}&=
    \left[
    \begin{array}{cccc}
    	(\frakm_D) & \ldots & \ldots & (\frakm_D) \\
    	(\calO_D) & \ddots & & \vdots \\
    	\vdots & \ddots & \ddots & \vdots \\
    	(\calO_D) & \ldots & (\calO_D) & (\frakm_D)
    \end{array}
    \right]^{(n_1,\dots,n_r)}\ .
    \end{align*}
    Then $\calO_D^{(n_1,\dots,n_r)}$ is an $R$-order in $\nMat{D}{n_1+\dots+n_r}$
    and its Jacobson radical is $J_D^{(n_1,\dots,n_r)}$.

    \begin{thm}
    \label{TH:structure-of-hered-rings}
        Let $D$ be a finite dimensional  division $F$-algebra
        and let $A$ be a hereditary $R$-order in $\nMat{D}{n}$.
        Then there are natural numbers $n_1,\dots,n_r$ with $\sum_in_i=n$ such that
        \[
        A\cong  \calO_{D}^{(n_1,\dots,n_r)} .
        \]
        Conversely, any $R$-order of this form is hereditary.
        The tuple $(n_1,\dots,n_r)$ is uniquely determined by $A$ up to a cyclic permutation.
    \end{thm}

	\begin{proof}
		When $\Cent(D)=F$, this is \cite[Theorem~39.14, Theorem~39.24]{Reiner_03_maximal_orders}.
		The general case follows by observing that the integral closure
		of $R$ in $\Cent(D)$ is itself a henselian DVR which is finitely
		generated as an $R$-module; see \cite[\S12,\S13]{Reiner_03_maximal_orders},
		for instance.
%
	\end{proof}

    For $A$ as in the theorem, we denote the equivalence class of $(n_1,\dots,n_r)$ under
    cyclic permutations by
    \[
    \inv(A).   
    \]
	The class $\inv(A)$ and the  simple $F$-algebra $A\otimes F$ determine
	$A$ up to isomorphism.
	
	For brevity, we shall write
	$(n_1,\dots,n_r)^s$ to denote
	the concatenation of $s$ copies of $(n_1,\dots,n_r)$, e.g.\ $(n)^s=(n,\dots,n)$ ($s$ times).

    \begin{prp}\label{PR:inv-over-Rsh}
    	Let $D$ be a central division $F$-algebra,   let $A=\calO_D^{(n_1,\dots,n_r)}$
    	and let
    	$k_D=\calO_D/\frakm_D$.
    	Let  $k''=\Cent(k_D)$ and let $k'$ denote the maximal 
    	subfield of $k''$ which is separable over $k$.
    	Write  $t=[k':k]$
    	and $s=\deg k_D = \sqrt{[k_D:k'']}$. Then
    	\[
		\inv(A\otimes {\sh{R}})=
		(sn_1,\dots,sn_r)^t    	
    	\]
    	In particular, if $\inv(A\otimes {\sh{R}})=(m_1,\dots,m_u)$,
    	then $\inv(A)=(\frac{m_1}{s},\dots,\frac{m_r}{s})$.
    \end{prp}

    \begin{proof}
	 	We first claim that $\inv(\calO_D\otimes \sh{R})=(s)^t$.
    	Let $k_{\mathrm{s}}$ denote the residue field of $\sh{R}$, which is also a separable
    	closure of $k$.
    	By Lemma~\ref{LM:hereditary-descent}(i), $\Jac(\calO_D\otimes \sh{R})=\frakm_D\otimes \sh{R}$,
    	hence 
    	\begin{equation}
    	\label{EQ:calO-quotient}
    	\frac{(\calO_D\otimes \sh{R})}{\Jac(\calO_D\otimes \sh{R})}\cong k_D\otimes \sh{R}\cong
    	k_D\otimes_k k_{\mathrm{s}} 
    	\cong k_D\otimes_{k''}(k''\otimes_{k'} (k'\otimes_{k}k_{\mathrm{s}})).
    	\end{equation} 
    	By assumption, $k'\otimes_k k_{\mathrm{s}}\cong k_{\mathrm{s}}\times \dots\times k_{\mathrm{s}}$
    	($t$ times).
    	Since $k''$ is a purely inseparable finite extension of $k'$, the tensor
    	product $\tilde{k}:=k''\otimes_{k'}k_{\mathrm{s}}$ is a separably closed field,
    	and hence $k_D\otimes_{k''}\tilde{k}\cong \nMat{\tilde{k}}{s}$.
    	Putting this into \eqref{EQ:calO-quotient},
    	we get
    	\begin{equation}
    	\label{EQ:D-mod-I}
    	\frac{\calO_D\otimes \sh{R}}{\Jac(\calO_D\otimes \sh{R})}\cong 
    	k_D\otimes_{k''}(\,\underbrace{\tilde{k}\times \dots\times \tilde{k}}_{t}\,)\cong
    	\underbrace{\nMat{\tilde{k}}{s}\times \dots\times\nMat{\tilde{k}}{s}}_{t} .
    	\end{equation}
		On the other hand, writing  $ \calO_D\otimes \sh{R} \cong \calO_E^{(m_1,\dots,m_u)}$
		for a suitable central division $\sh{F}$-algebra $E$,
		we see 
    	that 
    	\begin{equation}
    	\label{EQ:D-mod-II}
    	\frac{\calO_D\otimes \sh{R}}{\Jac(\calO_D\otimes \sh{R})}
    	\cong\frac{\calO_E^{(m_1,\dots,m_u)}}{J_E^{(m_1,\dots,m_u)}}
    	\cong  \nMat{k_E}{m_1}\times
    	\dots\times \nMat{k_E}{m_u}.
    	\end{equation}
    	The claim  
    	follows 
		by comparing the right hand sides of  
		\eqref{EQ:D-mod-I} and \eqref{EQ:D-mod-II}.

\smallskip    	
    	
    	We now prove the proposition. 
		The previous paragraph
		implies that $\calO_D\otimes \sh{R}\cong \calO_E^{(s)^t}$,
		and by Lemma~\ref{LM:hereditary-descent}(i),
		this isomorphism restricts to an isomorphism   
		$\frakm_D\otimes \sh{R}=J_E^{(s)^t}$.
		As a result, 
    	we have
    	\begin{equation*}\label{EQ:scalar-ext}
    	\calO_D^{(n_1,\dots,n_r)}\otimes \sh{R}\cong
    	\left[
		\begin{matrix}
    	(\calO_E^{(s,\dots,s)}) & (J_E^{(s,\dots,s)}) & \dots & (J_E^{(s,\dots,s)}) \\
    	(\calO_E^{(s,\dots,s)}) & (\calO_E^{(s,\dots,s)}) & \ddots & \vdots \\
    	\vdots & \ddots & \ddots & (J_E^{(s,\dots,s)}) \\
    	(\calO_E^{(s,\dots,s)}) &\dots  & \dots & (\calO_E^{(s,\dots,s)})
    	\end{matrix}
    	\right]^{(n_1,\dots,n_r)}\ .
    	\end{equation*}
    	Conjugating the right hand side by a suitable permutation matrix in $\nMat{E}{ts\sum_in_i}$
    	yields the proposition. Explicitly, writing $n=\sum_in_i$,
    	the required permutation
    	sends $i+s(j-1)+st(k-1)$ with $i\in \{1,\dots,s\}$, $j\in \{1,\dots,t\}$, $k\in \{1,\dots,n\}$,  
    	to $i+s(k-1)+sn(j-1)$.
    \end{proof}

    \begin{cor}\label{CR:Aut-A-henselian-case}
    	Suppose $R$ a henselian DVR and let $A$, $A'$ be two hereditary $R$-orders in central simple $F$-algebras.
		If $A$ and $A'$ become isomorphic after extending scalars to $F$ and some 
		faithfully flat \'etale $R$-algebra,
		then $A\cong A'$.
    \end{cor}

    \begin{proof}
    	Write $A\otimes F=\nMat{D}{n}$ and $A'\otimes F=\nMat{D'}{n'}$.
    	Since $A\otimes F\cong A'\otimes F$, we have $D\cong D'$. Since $A$ and $A'$ become isomorphic over a faithfully
    	flat \'etale $R$-algebra,
    	and since any faithfully flat \'etale $R$-algebra admits an $R$-algebra morphism
    	into $\sh{R}$,
    	we have $A\otimes  \sh{R}\cong A'\otimes  \sh{R}$. Now, by Proposition~\ref{PR:inv-over-Rsh},
    	$\inv(A)=\inv(A')$, so $A\cong A'$.
    \end{proof}

\subsection{Proof of Theorem~\ref{TH:Aut-A-works}}

	Recall that $R$ is assumed to be a semilocal Dedekind domain,
	$A$ is a hereditary $R$-order and $A'$ is any $R$-order.

		Arguing as in the proof of Corollary~\ref{CR:Aut-A-works}, 
		we see that $A'$ is hereditary.
		Now, Corollary~\ref{CR:Aut-A-henselian-case} implies that
		$A\otimes \hen{R_\frakp}\cong A'\otimes \hen{R_\frakp}$ for all 
		$\frakp\in\Max R$.
		We finish by using a patching argument to show that $A\cong A'$.
		
		Let $S=\prod_{\frakp\in \Max R} \hen{R_\frakp}$ and $K=\prod_{\frakp\in\Max R} \hen{F_\frakp}$.
		Then there are algebra isomorphisms $\phi:A\otimes F\cong A'\otimes F$ and $\psi:A\otimes S\cong A'\otimes S$.
		Consider $\psi^{-1}_K\phi_K:A\otimes K\to A\otimes K$. 
		Since $A\otimes F$ is a central simple $F$-algebra, the Skolem-Noether Theorem implies that there
		exists $a\in \units{(A\otimes K)}$ such that $\psi^{-1}_K\phi_K=\mathrm{Int}(a):=[x\mapsto axa^{-1}]$.
		A standard density argument (e.g.\ see the second paragraph in the proof of \cite[Theorem~5.1]{Bayer_17_weak_approx})
		implies that there are $b\in \units{(A\otimes F)}$, $c\in\units{(A\otimes S)}$ such
		that $a=c^{-1}b$. Then $\psi^{-1}_K\phi_K=\mathrm{Int}(c)^{-1}_K\mathrm{Int}(b)_K$,
		or rather
		$(\Int(c)\psi^{-1})_K=(\Int(b)\phi^{-1})_K$. Now, a patching argument (e.g.\ see
		\cite[Lemma~4.2(i)]{Bayer_17_weak_approx}) implies that there is an $R$-module isomorphism $\eta:A'\to A$
		such that $\eta_S=\Int(c)\psi^{-1}$ and $\eta_F=\Int(b)\phi^{-1}$. It is easy to see that $\eta$
		is an isomorphism of $R$-algebras, which completes the proof.
		\qed
	
\subsection{Counterexamples}

\label{subsec:counterex}	
	
	We finish with noting that   Theorem~\ref{TH:Aut-A-works} may fail
	if some of the assumptions are dropped.

\medskip

	We begin with observing that Theorem~\ref{TH:Aut-A-works} can fail for hereditary orders in simple, non-central, $F$-algebras.
	
	\begin{example}\label{EX:order-I}
		Let $R$ be a DVR with fraction field $F$.
		For brevity, denote the henselization of $R$ and its fraction
	field by $R'$ and $F'$, respectively.
	Suppose that there exists a cubic field extension $K/F$
	such that:
	\begin{itemize}
		\item $\Gal(K/F)=\{\id_K\}$, and
		\item $K\otimes F'\cong F'\times F'\times F'$.
	\end{itemize}
	(Explicit   choices with these properties are
	$R=\Z_{(7)}$, $F=\Q$, $K=\Q[\sqrt[3]{6}]$.)
	
	Let $\frakm'$ denote the maximal ideal of $R'$.
	Define $R'$-orders $B_1$ and $B_2$ in $\nMat{K}{2}\otimes F'\cong \nMat{F'}{2}^3$ as follows:
	\begin{align*}
	B_1&=
	\SMatII{R'}{\frakm'}{R'}{R'}\times
	\SMatII{R'}{\frakm'}{R'}{R'}\times
	\SMatII{R'}{R'}{R'}{R'}, \\
	B_2&=
	\SMatII{R'}{\frakm'}{R'}{R'}\times
	\SMatII{R'}{R'}{R'}{R'}\times
	\SMatII{R'}{\frakm'}{R'}{R'},
	\end{align*}
	and let
	\[
	A_i=\nMat{K}{2}\cap B_i\qquad(i=1,2).
	\]
	Observe that   $B_1$ and $B_2$ are   isomorphic as $R'$-algebras, but
	not as $R'\times R'\times R'$-algebras.
	
	It is easy to see that $A_1$ and $A_2$ are
	orders in $\nMat{K}{2}$ satisfying $A_i\otimes R'\cong B_i$ ($i=1,2$).
	Since $B_1$ and $B_2$ are hereditary (Theorem~\ref{TH:structure-of-hered-rings}),
	so are $A_1$ and $A_2$ (Corollary~\ref{CR:hereditary-descent}).
	
	Now, $A_1\otimes R'\cong A_2\otimes R'$ as $R'$-algebras.
	Since $R'$ is a direct limit of faithfully flat \'etale $R$-algebras,
	the $R$-orders $A_1$ and $A_2$ become isomorphic after base change to some  faithfully flat
	\'etale $R$-algebra. In addition, we have $A_1\otimes F\cong \nMat{K}{2}\cong A_2\otimes F$.
	However, we claim that $A_1\ncong A_2$.
	
	To see this, suppose $\vphi:A_1\to A_2$ is an isomorphism of $R$-algebras.
	Then $\vphi$ extends to an $F$-automorphism of $\nMat{K}{2}$.
	Since $\Gal(K/F)=\{\id_K\}$, the isomorphism $\vphi$ is $K$-linear,
	hence $\vphi_{R'}:A_1\otimes R'\to A_2\otimes R'$ is an $R'\times R'\times R'$-linear
	isomorphism, which cannot exist by our choice of the $R'$-orders $B_1$ and $B_2$.
	\end{example}
	
	In the previous example,
	the orders are isomorphic over
	the henselizations, but this fails to descend to the original ring.
	The next example shows that problems can also occur in passing from the strict
	henselization to the henselization
	if one allows orders in \emph{semisimple} $F$-algebras; compare
	with   Corollary~\ref{CR:Aut-A-henselian-case}.
	
	\begin{example}
		Suppose $R$ is a henselian DVR and $F$ admits a non-commutative central division algebra
		$D$ such that $D\otimes \sh{F}\cong \nMat{\sh{F}}{n}$. 
		(For example, take $R=\Z_p$, $F=\Q_p$
		and let $D$ be the unique quaternion division $F$-algebra.) 
		Using Corollary~\ref{CR:Aut-A-henselian-case},
		write $\inv(\calO_D\otimes \sh{R})=(s)^t$.
		Let 
		\[E=\nMat{D}{2}\times   \nMat{F}{2n}
		\]
		and define orders in $E$ as follows:
		\begin{align*}
		A_1&=\calO_D^{(1,1)}\times  \calO_F^{(2s)^{t}},\\
		A_2&=\calO_D^{(2)}\times  \calO_F^{(s)^{2t}}.
		\end{align*}
		Proposition~\ref{PR:inv-over-Rsh} implies that 
		$A_1\otimes \sh{R}\cong
		\calO_{\sh{F}}^{(s)^{2t}}\times \calO_{\sh{F}}^{(2s)^t}
		\cong
		A_2\otimes \sh{R}$,
		and $A_1\otimes F\cong A_2\otimes F$ is clear.
		Since $\sh{R}$ is a direct limit of faithfully flat \'etale $R$-algebras,
		this means that $A_1$ and $A_2$ become isomorphic over $F$ and some faithfully flat \'etale $R$-algebra.
		On the other hand, an
		$R$-algebra isomorphism $A_1\to A_2$
		will necessarily induce an isomorphism $\calO_D^{(1,1)} \to \calO_D^{(2)}$,
		which is impossible by Theorem~\ref{TH:structure-of-hered-rings}.
	\end{example}
	
	Finally, we note that Theorem~\ref{TH:Aut-A-works} may fail if $R$ is a Dedekind domain which is not semilocal.
	
	\begin{example}
		Let $R$ be a Dedekind domain whose class group is a nontrivial $2$-torsion group, let $I$ and
		$I'$ be two non-isomorphic fractional ideals of $R$, and let $A=\End_R(R\oplus I)$ and $A'=\End_R(R\oplus I')$.
		It is easy to check that the $R$-modules $R\oplus I$ and $R\oplus I'$ become isomorphic over $F$ and
		some faithfully flat \'etale $R$-algebra (e.g.  a suitable  
		Zariski covering of $\Spec R$), 
		so the same holds for $A$ and $A'$. However, $A\ncong A'$ as $R$-algebras.
		Indeed, for the sake of
		contradiction, assume that there is an isomorphism of $R$-algebra $\phi:A'\to A$. Then we may view $R\oplus I'$
		as a left $A$-module via $\phi$ and form $J:=\Hom_A(R\oplus I,R\oplus I')$. Morita Theory implies that $J$
		is a fractional ideal of $R$ satisfying $(R\oplus I)\otimes J\cong R\oplus I'$. However, the latter implies
		$IJ^2= I'J^2$ in $\mathrm{Cl}(R)$, which is impossible by our choice of $I$, $I'$ and the fact that $\mathrm{Cl}(R)$ is a $2$-torsion group.
	\end{example}

\section{$\uPGL_1(A)$-Torsors}

	Let $R$ an commutative ring and let $A$ be
	a  finite  projective $R$-algebra.
	As usual, let $\uGL_1(A)$ denote
	the $R$-group scheme determined by $\uGL_1(A)(S)=\units{(A\otimes S)}$.
	The group $R$-scheme $\uPGL_1(A)$
	is defined to be the $R$-scheme representing  the quotient sheaf 
	$\uGL_1(A)/\nGm{R}$ on the flat (fpqc)
	site of $\Spec R$. To see that
	it exists, apply \cite[XVI, Corollaire 2.3]{SGA3}
	to the morphism $\uGL_1(A)\to \uAut_R(A)$ sending
	a section to its corresponding inner automorphism.
	The group $\uPGL_1(A)$ 
	is smooth over $  R$
	by \cite[VI{${}_{\text{B}}$}, Proposition~9.2(xii)]{SGA3}, because
	$\uGL_1(A)$ is smooth and $\nGm{R}$ is flat over $R$.

	\begin{thm}\label{TH:PGL-A-works}
		Let $R$ be a regular semilocal domain with fraction field $F$ and let $A$ be a
		finite projective $R$-algebra. Then the restriction map
		\[
		\HH^1_\et(R,\uPGL_1(A))\to \HH^1_\et(F,\uPGL_1(A))
		\]
		has trivial kernel.
	\end{thm}
	
	The injectivity of $\HH^1_\et(R,\uPGL_1(A))\to \HH^1_\et(F,\uPGL_1(A))$
	is more delicate and so far we were unable to establish it.

\begin{proof}
We first observe that since $\uPGL_1(A)$ is smooth over $R$,
the canonical map $\HH^1_{\et}(R,\uPGL_1(A))\to \HH^1_{\fpqc}(R,\uPGL_1(A))$
is an isomorphism; see \cite[Th\'eor\`eme~11.7(1), Remarque~11.8(3)]{Groth_68_Brauer_III}.
It is therefore enough to prove the theorem with $\fpqc$ cohomology.
To this end, 
we first establish two independent facts.

\medskip

First, we show that the set $\HH^1_\fpqc(R,\uGL_1(A))$ is trivial.  
The flat cohomology set $\HH^1_\fpqc(R,\uGL_1(A))$ parametrizes the isomorphism classes of left $A$-modules $P$ which become isomorphic to $A$ itself after a faithfully flat   extension of $R$. As such an $A$-module $P$ is necessarily a finitely generated and projective $A$-module 
(see the first paragraph of the proof
of \cite[Proposition~5.1]{Bayer_17_rational_iso_herm}), one can apply 
\cite[Proposition 2.11]{Bayer_17_rational_iso_herm} to conclude that $P$ is isomorphic to $A$. Consequently, there is only one such isomorphism class of $A$-modules and the set $\HH^1_\fpqc(R,\uGL_1(A))$ is a singleton.

\medskip

Second, as $R$ is a regular domain, the morphism  $\HH^2_\fpqc(R,\uGm)\to \HH^2_\fpqc(F,\uGm)$ is injective by \cite[Corollaire 1.8]{Groth_68_Brauer_II}. 

\medskip

Now,
consider  the exact sequence 
$$1\to \nGm{R} \to \uGL_1(A) \to \uPGL_1(A) \to 1$$
of sheaves  on the flat (fpqc) site of $\Spec R$. 
It induces the following diagram of pointed cohomology sets with exact rows.
\[
\xymatrix{
\HH^1_\fpqc(R,\uGL_1(A))  \ar[r]\ar[d] &
\HH^1_\fpqc(R,\uPGL_1(A)) \ar[r]\ar[d] &
\HH^2_\fpqc(R,\uGm) \ar[d] \\
\HH^1_\fpqc(F,\uGL_1(A)) \ar[r] &
\HH^1_\fpqc(F,\uPGL_1(A)) \ar[r] &
\HH^2_\fpqc(F,\uGm)
}
\]
A straightforward diagram chasing using the two facts proved above gives the result.
\end{proof}

\begin{remark}
	Theorem~\ref{TH:Aut-A-works}
	and Theorem~\ref{TH:PGL-A-works} require independent proofs
	because
	in general, the  morphism $\uPGL_1(A)\to \uAut_R(A)$
	sending a section $x$ to conjugation by $x$ is not an isomorphism,
	even when $A$ is a hereditary $R$-order in a central simple
	$F$-algebra. 
	For example, let $R$ be a DVR with maximal ideal $\frakm=\pi R$
	and let $A=\smallSMatII{R}{\frakm}{R}{R}$.
	Then $a\mapsto \smallSMatII{0}{\pi}{1}{0}a\smallSMatII{0}{\pi}{1}{0}^{-1}$
	is an automorphism of $A$ which is not inner.
\end{remark}

\section{Hereditary Orders with Involution}

\label{sec:counterexamples}

	Throughout, $R$ is a semilocal Dedekind domain with fraction field $F$.
	As in Section~\ref{sec:orders}, unadorned tensors are taken over $R$.
	
	We now show that for suitably chosen $R$, there exist non-isomorphic  hereditary $R$-orders with involution
	which become isomorphic (as algebras with involution) over $F$
	and over a faithfully flat \'etale $R$-algebra.
	We shall see that the  counterexample can be chosen so that 
	base change of the involution to $F$ is either
	orthogonal, symplectic or unitary. 
	However, this cannot happen for Azumaya orders with involution,
	at least when $R$ is local, by \cite{Panin_05_purity_for_mult}.

	\begin{example}\label{EX:main-counterex}
	Let $R=\R[\![t]\!]$.
	The complex conjugation on $\C$ and the canonical
	symplectic involution on the real quaternions $\bbH$
	are both denoted $\overline{\phantom{w}}$. 	
	Define the $R$-order with involution $(\calO,\tau)$ to be any of the following:
	\begin{enumerate}
		\item $\calO=R$ and $\tau=\id_R$,
		\item $\calO=R\otimes_{\R} \C=\C[\![t]\!]$ and $\tau=\id_R\otimes \quo{\phantom{w}}$,
		\item $\calO=R\otimes_{\R} \bbH=\bbH[\![t]\!]$ and $\tau=\id_R\otimes \quo{\phantom{w}}$.
	\end{enumerate}
	Theorem~\ref{TH:structure-hered-div} implies that $\calO$ is the unique
	maximal $R$-order in $D:=\calO\otimes_R \R(\!(t)\!)$,
	which is either $\R(\!(t)\!)$, $\C(\!(t)\!)$ or $\bbH(\!(t)\!)$.
	
	Let $\frakm=t\calO$ denote the maximal ideal of $\calO$,
	and, with the notation of \ref{subsec:henselian-case},
	define
	\[A=\calO_D^{(4,2)}=\SMatII{\calO}{\frakm}{\calO}{\calO}^{(4,2)}\ .\]
	By Theorem~\ref{TH:structure-of-hered-rings},  $A$ is a hereditary $R$-order in $\nMat{D}{6}$.
	
	The involution $\tau$ induces an involution $(\alpha_{ij})_{i,j}\mapsto
	(\tau(\alpha_{ji}))_{i,j}: \nMat{D}{6}\to \nMat{D}{6}$, which we also denote
	by $\tau$.
	Let
	\begin{align*}
	a_1&=\mathrm{diag}(1,-1,1,-1,t,\phantom{-}t),\\
	a_2&=\mathrm{diag}(1,-1,1,\phantom{-}1,t,-t),
	\end{align*}	
	and define involutions $\sigma_1,\sigma_2:A\to A$
	by 
	\[\sigma_i(x)=a_i^{-1}\tau(x)a_i\qquad(i=1,2).\]
	We claim that $(A,\sigma_1)$ and $(A,\sigma_2)$ become isomorphic
	after tensoring with $F$ and some faithfully flat \'etale $R$-algebra, 
	but are nevertheless non-isomorphic as $R$-algebras with involution.
	
	To see this, we note that if $S$ is an $R$-algebra and $u\in A\otimes S$	
	satisfies $\tau(u)a_2u=\alpha a_1$ for some $\alpha\in\units{S}$, then $x\mapsto uxu^{-1}$
	determines an isomorphism $(A\otimes S,\sigma_1\otimes\id_S)\to (A\otimes S,\sigma_2\otimes\id_S)$.
	For $S=F$, one can take
	\[
	u=
	\left[\begin{matrix}
		{\frac{1}{2}t+\frac{1}{2}} &
		{\frac{1}{2}t-\frac{1}{2}} \\[2pt]
		{\frac{1}{2}t-\frac{1}{2}} &
		{\frac{1}{2}t+\frac{1}{2}}
		\end{matrix}\right]
	\oplus
	\left[\begin{matrix}
		& & t & \\
		& & & t\\
		1 & \\
		& 1
		\end{matrix}\right]
	\]	
	with $\alpha=t$. For the faithfully
	flat \'etale $R$-algebra $S=R[\sqrt{-1}]$, one can take	
	\[
	u=
	\left[\begin{matrix}
		1 & \\
		& 1
		\end{matrix}\right]
	\oplus
	\left[\begin{matrix}
		1 & \\
		& \sqrt{-1}
		\end{matrix}\right]
	\oplus
	\left[\begin{matrix}
		1 & \\
		& \sqrt{-1}
		\end{matrix}\right]
	\]	
	with $\alpha=1$.
	
	We proceed by showing that
	there is no  isomorphism $(A,\sigma_1)\to (A,\sigma_2)$.
	Let $\quo{A}:=A/\Jac(A)\cong \nMat{k_D}{4}\times \nMat{k_D}{2}$,
	where $k_D=\calO/\frakm$ is either $\R$, $\C$ or $\bbH$.
	For $i=1,2$, the involution $\sigma_i$ induces an involution $\quo{\sigma}_i:\quo{A}\to \quo{A}$,
	and direct computation shows that
	$(\quo{A},\quo{\sigma}_i)=(\nMat{k_D}{4},\sigma'_i)\times (\nMat{k_D}{2},\sigma''_i)$
	for suitable $\R$-involutions $\sigma'_i$, $\sigma''_i$.
	In fact, $\sigma''_1$ and $\sigma''_2$
	are given by
	\[\sigma''_1[\begin{smallmatrix} x& y \\ z & w\end{smallmatrix}]=
	[\begin{smallmatrix} \quo{x}& \quo{z} \\ \quo{y} & \quo{w}\end{smallmatrix}]
	\qquad\text{and}\qquad
	\sigma''_2[\begin{smallmatrix} x& y \\ z & w\end{smallmatrix}]=
	[\begin{smallmatrix} 1& 0 \\ 0 & -1\end{smallmatrix}]^{-1}
	[\begin{smallmatrix} \quo{x}& \quo{z} \\ \quo{y} & \quo{w}\end{smallmatrix}]
	[\begin{smallmatrix} 1& 0 \\ 0 & -1\end{smallmatrix}]\ .
	\]
	
	An isomorphism $\psi:(A,\sigma_1)\to (A,\sigma_2)$
	would induce an isomorphism of $\R$-algebras
	with involution
	$\quo{\psi}:(\quo{A},\quo{\sigma}_1)\to (\quo{A},\quo{\sigma}_2)$.
	Considering the action on $\quo{\psi}$ on the central
	idempotents of $\quo{A}$, we see that
	$\quo{\psi}$ must further restrict
	to an isomorphism
	$(\nMat{k_D}{2},\sigma''_1)\to (\nMat{k_D}{2},\sigma''_2)$.
	However, $\sigma''_1$ is easily seen to be anisotropic (meaning that $\sigma''_1(x)x=0$
	implies $x=0$) while $\sigma''_2$ is isotropic, so $\psi$ cannot exist.

	Let $\uAut_R(A,\sigma_1)$ denote 
	the group $R$-scheme representing
	the functor 
	$S\mapsto  \Aut_S(A\otimes S,\sigma_1\otimes\id_S)$.
	From the previous discussion, we conclude that the restriction map
	\[
	\HH^1_\et(R,\uAut_R(A,\sigma_1))\to \HH^1_\et(F,\uAut_R(A,\sigma_1))
	\]
	is not injective. Indeed, the left hand side
	classifies $R$-orders with involution which become isomorphic
	to $(A,\sigma_1)$ over some faithfully flat \'etale $R$-algebra,
	and
	$(A,\sigma_2)$ represents a nontrivial class which is mapped
	to the trivial element
	of $\HH^1_{\et}(F,\uAut_R(A,\sigma_1))$.
	
	We finally note that in cases (1), (2) and (3) above, 
	$\sigma_1\otimes\id_F:A\otimes F\to A\otimes F$
	is  orthogonal, unitary and symplectic, respectively. 
	\end{example}
	
	Call an $R$-order with involution $(A,\sigma)$ \emph{residually anisotropic}
	if the induced involution $\quo{\sigma}:A/\Jac(A)\to A/\Jac(A)$
	is anisotropic, i.e.\ $\quo{\sigma}(x)x=0$ implies $x=0$ for all $x\in A/\Jac(A)$.
	The reader will notice that in Example~\ref{EX:main-counterex},
	both   $(A,\sigma_1)$ and $(A,\sigma_2)$
	are not residually anisotropic, and the isotropicity plays a crucial role.
	We therefore ask:
	
	\begin{que}\label{QE:res-aniso}
		Let $R$ be a semilocal Dedekind domain and let $(A,\sigma)$,
		$(A',\sigma')$ be two hereditary $R$-orders with involution which
		become isomorphic over the fraction field of $R$ and over some faithfuly
		flat \'etale $R$-algebra. Suppose $(A,\sigma)$ is residually anisotropic.
		Is $(A,\sigma)$ isomorphic to $(A',\sigma')$ as $R$-algebras with involution?
	\end{que}
	
	The question is also motivated by the work  of Bruhat and Tits on the cohomology
	of reductive groups over henselian discretely valued fields \cite{Bruhat_87_reductive_groups_III}.
	Specifically, it seems   likely that a positive answer
	should follow from \cite[Lemma~3.9]{Bruhat_87_reductive_groups_III} in
	case $R$ is a complete DVR; we elaborate about this in the next section.
	That said, we hope that a more
	direct proof can be found.

\section{Discussion}

\label{sec:remarks}

	We finish with explaining how the results of the previous sections
	relate to a question asked by the first two authors, \cite[Question~6.4]{Bayer_17_rational_iso_herm}.
	
\medskip

	Let $R$ be a semilocal Dedekind domain with fraction field $F$
	and suppose that $R/\frakp$ is perfect for all $\frakp\in \Max R$.
	We let $\hat{R}_\frakp$ and   $\hat{F}_\frakp$ denote the completion of $R_\frakp$
	and its corresponding fraction field.
	We shall use the notation of Section~\ref{sec:orders} for henselizations and strict henselizations.

	Let $\calG$ be a group scheme over $R$
	and let $\bfG=\calG\times_R F$ denote its generic fiber, which we assume to be
	reductive and connected.
	In \cite[\S6]{Bayer_17_rational_iso_herm},
	the group scheme $\calG$
	was called a \emph{point-stablizer} (resp.\ \emph{parahoric})
	group scheme for $\bfG$
	if for every $\frakp\in \Max R$,
	the group $\hat{R}_\frakp$-scheme $\calG \times_{R} \hat{R}_\frakp$
	coincides with one of  the group
	schemes that Bruhat and Tits \cite{Bruhat_84_reductive_groups_II}
	associate with  stablizers of points of the affine
	building of  $\bfG\times_F \hat{F}_\frakp$  (resp.\ parahoric subgroups
	of  $\bfG(\hat{F}_\frakp)$).
	Briefly, letting $\calB(\bfG,\hat{F}_\frakp)$ denote the (extended) affine Bruhat--Tits
	building of $\bfG\times_F \hat{F}_\frakp$,
	the point stablizer group scheme associated with a point $y\in \calB(\bfG,\hat{F}_\frakp)$
	is the smooth affine group $\hat{R}_\frakp$-scheme $\calG_y$
	with generic fiber $\bfG\times_F\hat{F}_\frakp$
	characterized by the condition that $\calG_y(\hat{R}^{\mathrm{sh}}_\frakp)$
	is the stablizier of $y\in \calB(\bfG,\hat{F}^{\mathrm{sh}}_\frakp)$
	under the action of $\bfG(\hat{R}^{\mathrm{sh}}_\frakp)$.
	The neutral component of $\calG_y\to \Spec \hat{R}_\frakp$ is then called the parahoric group
	scheme 	associated with $y$.
	
\medskip
	
	Let $A$ be a hereditary $R$-order in a simple $F$-algebra 
	and let $\sigma:A\to A$ be an $R$-involution.
	We write $A_F=A\otimes F$, $\sigma_F =\sigma\otimes\id_F$
	and let $\uU(A,\sigma)$ denote the affine group $R$-scheme determined
	by $\uU(A,\sigma)(S)=U(A_S,\sigma_S):=\{a\in A_S\where a^\sigma a=1\}$
	for any $R$-algebra $S$.
	In order to guarantee that
	$\bfG:=\uU(A_F,\sigma_F)\to \Spec F$ is connected,
	we assume further that either $\Cent(A_F)=F$
	and $\sigma_F$ is symplectic, or
	$\Cent(A_F)$ is a quadratic \'etale $F$-algebra
	and $\sigma_F$ is unitary.
	The excluded case where $\Cent(A_F)=F$
	and $\sigma_F$ is orthogonal can be handled similarly 
	after few modifications.

	In \cite[\S5--6]{Bayer_17_rational_iso_herm}, the first two authors
	showed that
	\[
	\HH^1_\et(R,\uU(A,\sigma))\to \HH^1_\et(F,\uU(A,\sigma))
	\]
	is injective,
	and moreover, that the point stablizer group schemes of   $\bfG$
	are all of the form $\uU(A,\sigma)$ 
	as $A$ varies over the hereditary orders in $A_F$ stable under $\sigma_F$.
	This has led the first two authors to ask whether
	$\HH^1_\et(R,\calG)\to \HH^1_\et(F,\calG)$ is injective
	for any point stablizer (resp.\ parahoric) group scheme $\calG$ of a connected
	reductive group scheme over $F$, \cite[Question 6.4]{Bayer_17_rational_iso_herm}.
	We explain briefly how the results and counterexamples of the previous sections
	relate to this question.

	The generic fibers of   $\uAut_R(A,\sigma)$ and $\uAut_R(A)$
	are the    reductive  groups $\uAut_F(A_F,\sigma_F)$
	and $\uAut_F(A_F)$ (which is just $\uPGL_1(A_F)$ if $\Cent(A)=F$),
	the buildings of which are well-understood; see \cite{Abram_02_lattice_models_building} or
	\cite{Bruhat_84_building_GLD}, \cite{Bruhat_87_buidling_SU}, for instance.
	\emph{Provided that  $\uAut_R(A,\sigma)$ and $\uAut_R(A)$
	are smooth over $R$}, one can use these sources to show that 
	these are indeed point-stablizer group
	schemes of their corresponding generic fibers. Then,
	the results of this paper   give a mixed answer to 
	the question above:
	By Theorem~\ref{TH:Aut-A-works}, the answer is positive for some of the point-stablizier group schemes 
	of $\uAut_F(A_F)=\uPGL_1(A_F)$ when $\Cent(A)=F$, whereas 
	by Example~\ref{EX:main-counterex}, the answer can be negative for some
	point-stabilizer group schemes of $\uAut_F(A_F,\sigma_F)$.
	If moreover $\uPGL_1(A)\to \Spec R$
	is the 	neutral component of $\uAut_R(A)\to \Spec R$ when $\Cent(A)=F$, then Theorem~\ref{TH:PGL-A-works}
	provides a partial positive answer for the question in the case of   parahoric group
	schemes of $\uPGL_1(A_F)$.
	
	The examples in Subsection~\ref{subsec:counterex} do not
	relate to \cite[Question 6.4]{Bayer_17_rational_iso_herm} because,
	in these examples, the generic fiber of $\uAut_R(A)$ is reductive but not connected.
	
	We remark that the smoothness of $\uAut_R(A)$ and $\uAut_R(A,\sigma)$ over $R$
	as well as the condition that $\uPGL_1(A)\to\Spec R$ is the neutral component
	of $\uAut_R(A)\to \Spec R$
	are   false for arbitrary $R$-orders in central simple algebras.
	We hope to address both of these problems in the case of hereditary
	orders in   subsequent work.
	
\medskip

	We finally note that in the case where $R$ is a complete DVR
	with residue field $k$,
	\cite[Question~6.4]{Bayer_17_rational_iso_herm} reduces to asking whether
	a result of Bruhat and Tits, \cite[Lemma~3.9]{Bruhat_87_reductive_groups_III},
	can be strengthened.
	Indeed, the latter result
	asserts the injectivity of $\HH^1_\et(R,\calG_y)\to \HH^1_\et(F,\calG_y)$
	for any  point stablizer group scheme $\calG_y$ 
	in which  $y$ is the barycenter of a cell in the affine building
	of $\bfG:=\calG_y\times_R F$ and such that the closed fiber
	$\calG_y\times_R k$ 
	is almost anisotropic in the sense that it has no proper parabolic subgroups.
	
	Since in Example~\ref{EX:main-counterex}, the base ring is a complete DVR,
	it is possible that the reason for the apparent negative
	answer to \cite[Question~6.4]{Bayer_17_rational_iso_herm}
	is  that arbitrary point-stablizer group schemes are allowed.
	Returning to the case where $R$ is any semilocal Dedekind domain with perfect
	residue fields, one can   ask instead whether
	$\HH^1_\et(R,\calG )\to \HH^1_\et(F,\calG )$
	is injective when $\calG\to \Spec R$ is \emph{barycentric} point-stablizier group
	scheme  whose closed fibers are almost anisotropic.
	This statement already follows from \cite[Lemma~3.9]{Bruhat_87_reductive_groups_III} when
	$R$ is a complete DVR.

\bibliographystyle{plain}
\bibliography{orders}

\end{document}